\documentclass{article}
\usepackage{amsthm}
\usepackage{amsmath}
\usepackage{hyperref}
\usepackage{blindtext}
\usepackage{amsfonts}
\usepackage[letterpaper,body={13cm,21cm}, mag=1000]{geometry}
\usepackage{amssymb}
\usepackage{secdot}
\usepackage{setspace}
\usepackage{titlesec}
\numberwithin{equation}{section}
    
\titleformat{\section}[runin]{\bfseries\filcenter}{\thesection}{1em}{}

\renewcommand{\thesection}{\arabic{section}}
\title{\large \bf Direct and Inverse Problems in Baumslag-Solitar Group $BS(1,3)$}
\author{\small \bf Sandeep Singh and  Ramandeep Kaur\\
\small \em Department of Mathematics\\
\small \em Akal University, Talwandi Sabo - 151302,
India\\
\small E-mail:  sandeepinsan86@gmail.com, ramandeepka71@gmail.com
} 
 
\date{}
\newtheorem{thm}{Theorem}[section]

\newtheorem{cor}[thm]{Corollary}

\theoremstyle{plain}
\newtheorem{theorem}{Theorem}[section]

\newtheorem{lemma}[theorem]{Lemma}

\begin{document}
\newcommand{\Rn}[1]{\romannumeral#1\relax}
\newcommand{\RN}[1]{\uppercase\expandafter{\romannumeral#1}}
\maketitle
\begin{abstract}
\noindent {\bf Abstract.} 
For integers $m$ and $n$, the Baumslag-Solitar groups, denoted as $BS(m,n)$, are groups generated by two elements with a single defining relation: $BS(m,n) = \langle a, b | a^mb=ba^n\rangle$. The sum of dilates, denoted as $r \cdot A + s \cdot B$ for integers $r$ and $s$, is defined as $\{ra + sb; a\in A, b\in B\}$. In 2014, Freiman et al. \cite{freiman} derived direct and inverse results for sums of dilates and applied these findings to address specific direct and inverse problems within Baumslag-Solitar groups, assuming suitable small doubling properties. In 2015, Freiman et al. \cite{freiman15} tackled the general problem of small doubling types in a monoid, a subset of the Baumslag-Solitar group $BS(1,2)$. This paper extends these investigations to solve the analogous problem for the Baumslag-Solitar group $BS(1,3)$.

\end{abstract}
     
\vspace{2ex}


\vspace{2ex}

\section{Introduction}

Combinatorial number theory explores the sum-product problems of sets, constituting a dynamic field within additive combinatorics. The sum-product problem involves direct and inverse problems in additive number theory, where direct problems focus on describing the size of sumsets associated with a given set, and inverse problems involve determining the structure of sets based on the cardinality of sumsets obtained from direct problems. Let $G$ be an abelian group, and let $A$ and $B$ be finite subsets of $G$. The sumset of $A$ and $B$ is defined as $A+B=\{a+b:a\in A,\, b\in B\}$ and product set $A*B=\{a*b;a\in A,b\in B\}.$ For $r$ and $s$ be integers, sum of dilates is defined as $r\cdot A+s\cdot B=\{ra+sb|\,a\in A,\,b\in B\}.$ Historically, in 1813, Cauchy \cite{C} established a significant result for the cyclic group $G=\mathbb{Z}/{p\mathbb{Z}}$ of prime order, demonstrating that for subsets $A$ and $B$ of $\mathbb{Z}/{p\mathbb{Z}}$, $|A+B|\geq \min\{p, |A|+|B|-1\}$. Davenport \cite{D} rediscovered this result in 1935, now known as the Cauchy-Davenport theorem. Subsequent researchers \cite{EK12, EK13, EK14, EK15, EKP,  P} have explored sumset problems in various groups, including abelian, non-abelian, solvable, non-solvable, and dihedral groups. Freiman et al. \cite{freiman} studied the sum of dilates in  Baumslag-Solitar (non-abelian) groups. They also obtained extended inverse problems in Baumslag-Solitar (non-abelian) groups.
For any set $A$, if the inequality $|A^2|\leq \alpha |A|+\beta$ holds for some small $\alpha\geq 1$ and small $|\beta|$, then it is called inverse problems of small doubling type. Freiman et al. \cite{freiman15} studied the inverse problems of small doubling type for Baumslag-Solitar group $BS(1,2)$. They proved the following result:
\begin{theorem} [\cite{freiman15}]
If $S$ is a finite non-abelian subset of $BS^{+}(1,2)$ of size $|S|=k$, then the following statements hold:\\
(a) The size of $S^2$ satisfies  $|S^2|\geq 3k-2.$\\
(b) If $|S^2|=(3k-2)+h<\frac{7}{2}k-4,$ then there exists a finite set of integers $A\subseteq \mathbb{Z}$ such that $S=ba^A$ and the set $A$ is contained in an arithmetic progression of size $k+h<\frac{3}{2}k-2.$
\end{theorem}

\noindent Let $S$ be a finite subset of $BS(1,n)$ of size $k$ and suppose that $S$ is contained in the coset $b^r\langle a\rangle$ of $\langle a \rangle$ in $BS(1,n)$ for some positive integer $r.$ Then for $A=\{x_1,x_2,\ldots,x_k\}$, 
$S=\{b^ra^{x_1},b^ra^{x_2},\ldots,b^ra^{x_k}\}=b^ra^A.$ For $BS(1,n)$, observe that 
\begin{align*}
(b^ra^A)(b^pa^B) 
&=b^r(b^pa^{n^p\cdot A})a^B\\ 
&=b^{r+p}a^{n^p\cdot A+B};\, r,p \in \mathbb{N}.
\end{align*} 
\noindent In this paper, we determine the structure of subset of $BS(1,3)$, satisfying some small doubling condition. We also solve it for subsets of monoid $BS^{+}(1,3)=\{b^ma^x\in BS(1,3)\},$ where $x$ is an arbitrary integer and $m$ is a non negative integer. Particularly, we have proved  the following result:
\begin{theorem}\label{thm3}
Let $S$ be a finite non-abelian subset of $BS^{+}(1,3)$ of size $k=|S|\geq 3.$ Then \\
(i) If $S=S_0\cup S_1\cup\cdots\cup S_t,$ where $t\geq 1,$  then $|S^2|\geq \frac{7}{2}|S|-6.$\\
(ii) If $S=b^ra^A$, where $A$ is a finite set of integers containing $0$, $|A|=k$ and $|S^2|=(t+2)k-2t+h<\frac{7}{2}k-6$, where $c_3(A)=t$ denotes the number of distinct classes of $A$ modulo $3$. Then $A$ is contained in an arithmetic progression of size $k+h<\frac{5}{2}k-t(k-2)-6.$
 \end{theorem} 

\noindent Let $c_m(A)$ denotes the number of distinct classes of $A$ modulo $m$.
Consider the set  $A=\{a_0<a_1<\cdots<a_{r-1}\}.$ Let $\ell(A)=\max(A)-\min(A)=a_{r-1}-a_{0}$. 
$h_A=\ell(A)+1-|A|,$ denotes the number of holes in set $A.$ By $d(A)$, we denote
$d=d(A)=\gcd(a_1-a_0,a_2-a_0,\ldots,a_{r-1}-a_0).$ Let $A^{*}=\frac{1}{d}(A-\min(A))$, $\ell^{*}=\ell(A^{*})=\max(A^{*})=\frac{\ell}{d}.$\\

\noindent We will use the following results frequently:
\begin{theorem}\label{thm1}
Let $A\subseteq \mathbb{Z}$ be finite subsets such that $C_3(A)=t$; $0\in A.$ If $|A+3\cdot A|= (t+1)|A|-t+h\leq (t+2)|A|-2t$ for some integer $h$, then $A$ is subset of arithmetic progression of length at most $|A|+h\leq 2|A|-3.$ 
\end{theorem}\label{thm4}
  
\begin{theorem}[{\cite{LS},\cite{S}}]\label{LSS} Let $A$ and $B$ be finite subsets of $\mathbb{N}$ such that $0\in A\cap B.$ Define 
 $\delta_{A,B}=
 \begin{cases}
 1 & if\, \ell(A)=\ell(B)\\
 0 & if\, \ell(A)\neq \ell(B).
 \end{cases}
 $
 Then the followings hold:
 \begin{enumerate}
     \item If $\ell(A)=\max(\ell(A),\ell(B))\geq |A|+|B|-1-\delta_{A,B}$ and $d(A)=1,$ then $$|A+B|\geq |A|+2|B|-2-\delta_{A,B}$$
     \item If $\max(\ell(A),\ell(B))\leq |A|+|B|-2-\delta_{A,B}$, then
     $$|A+B|\geq \max(\ell(A)+|B|,\ell(B)+|A|).$$
 \end{enumerate}
 \end{theorem}
 
\begin{theorem}[\cite{freiman}]\label{thm2}
Let $A$ be a finite subset of integers. Then the following statement hold:\\
(a) If $S=ba^A\subseteq BS(1,3),$ then $|S^2|\geq 4|S|-4$ and $|S^2|=4|S|-4$, iff either one of the following holds: $A=\{0,1,3\}$, $A=\{0,1,4\}$, $A=3\cdot \{0,1,\ldots,n\}\cup(3\cdot \{0,1,\ldots,n\}+1)$ or $A$ is affine transform of one of these sets.
\end{theorem}
\begin{cor}[\cite{freiman}]\label{cor1}
Let $S$ be a finite subset of $BS(1,r)$ of size $k\geq 1$. Suppose $r\geq 3$ and $S=ba^{A},$ where $A$ is a finite set of integers. Then 
$$|S^2|=|A+r\cdot A|\geq \max(4k-4,1)\geq 3k-2.$$ 
\end{cor}
\begin{theorem}[\cite{freiman15}]\label{thm5}
For $n\geq 2,$ the group 
$BS(1,n)=\langle a,b|ab=ba^n\rangle $ is orderable.
\end{theorem}

\section{Some Important Lemmas}
The following lemmas will be used to prove Theorem \ref{thm3}
\begin{lemma}\label{l1}
Let $S\subset BS^{+}(1,3)$ be a finite set of size $k.$  Then the followings hold:\\
(a) For $t\geq 1$ and $1\leq j\leq t$ such that $k_j=|S_j|\geq 2$, then the set $S$ generates a non abelian group.\\
(b) For $k=|S|\geq 4$, $S=U\cup V$ where $U=b^ma^M$ and $V=b^na^N$,  $0\leq m<n$ are two integers and $M$, $N\subseteq \mathbb{Z}$, then $|S^2|\geq \frac{7}{2}|S|-6.$\\
\end{lemma} 
\begin{proof}
(a) We will divide the proof of lemma in two cases based on the value of $j.$ 
 \noindent {\bf{Case 1:}} Initially, we assume that $j=0.$\\
If $m_0=0$, then $k_0=|S_0|=|A_0|\geq 2$ implying  $S_0\neq \{1\}$ and $A_0\neq \{0\}.$ Since $t\geq 1,$ there exist three integers $m,x,z$ such that $m\geq 1$, $x\neq 0$, $a^x\in S_0$ and $b^ma^z\in S_1.$
Since $a^x(b^ma^z)=b^ma^{z+3^mx}\neq (b^ma^z)a^x=b^ma^{z+x}.$
Thus, $S$ generates a non-abelian group.\\
 If $m_0\geq 1$, then $k_0=|S_0|=|b^{m_0}a^{A_0}|\geq 2$. Thus $|A_0|\geq 2.$  So there are two integers $x\neq y$ such that $\{b^{m_0}a^x,b^{m_0}a^y\}\subseteq S_0.$ Hence $(b^{m_0}a^x)(b^{m_0}a^y)=b^{m_0}a^{y+3^{m_0}x}\neq (b^{m_0}a^y)(b^{m_0}a^x)=b^{2m_0}a^{x+3^{m_0}y},$ because  $x\neq y$ and $m_0\geq 1.$

\noindent {\bf{Case 2:}} Next, let's consider the case where $j\geq 1.$\\
Since  $j\geq 1$, therefore $m_j\geq 1$ and $k_j=|S_j|=|b^{m_j}a^{A_j}|\geq 2$ implies that $|A_j|\geq 2.$  Consequently, there exist two integers $x\neq y$ such that $\{b^{m_j}a^x,b^{m_j}a^y\}\subseteq S_j.$ Hence $$(b^{m_j}a^x)(b^{m_j}a^y)=b^{m_j}a^{y+3^{m_j}x}\neq (b^{m_j}a^y)(b^{m_j}a^x)=b^{2m_j}a^{x+3^{m_j}y},$$ as $x\neq y$ and $m_j\geq 1.$ This completes the proof.\\

\noindent (b) Let $S=U\cup V$. Observe that $U\cap V=\phi$ implies $|S|=|U|+|V|.$ Also, we have $S^2=U^2\cup (UV\cup VU)\cup V^2$, where $U^2=b^{2m} a^{M+3^m\cdot M}$, $V^2=b^{2n} a^{N+3^n\cdot N},$ 
$UV=(a^M b)(a^N b)=b^{m+n}a^{N+3^n\cdot M}$, $VU=(a^N b)(a^M b)=b^{m+n}a^{M+3^m\cdot N}.$ Thus 
$|U^2|=|M+3^m\cdot M|\geq 4|M|-4$, $|V^2|=|N+3^n\cdot N|\geq 4|N|-4$.\\
$|UV|=|M+N|\geq |M|+|N|-1$ and  $|VU|=|N+M|\geq |N|+|M|-1.$\\
 Since  the sets $b^{2m}a^{\mathbb{Z}}$, $b^{2n}a^{\mathbb{Z}}$ and $b^{m+n}a^{\mathbb{Z}}$ are pairwise disjoint,  we can express
 $$|S^2|=|U^2|+|UV\cup VU|+|V^2|.$$ 
\noindent Now, let's consider two cases:

\noindent {\bf{Case 1:}} First, assume that $m=0.$\\
Let $S=U\cup V$, where $U=a^M$, $V=b^na^N$, and $n\geq 1.$ Additionally, $S^2=U^2\cup (UV\cup VU)\cup V^2.$ To determine $|S^2|$, we must compute $|U^2|$, $|V^2|$, and $|UV\cup VU|$.\ We have $U^2=a^{M+M}$, $V^2=b^{2n}a^{N+3^n\cdot N}$, $UV=b^na^{N+3^n\cdot M}$, and $VU=b^na^{M+N}.$ Therefore, $|U^2|=|M+M|$ and, by Corollary \ref{cor1}, $|V^2|=|N+3^n\cdot N|\geq 4|N|-4$. Also,
\begin{align*} |UV\cup VU|
&=|(N+3^n\cdot M)\cup (M+N)|=|(M\cup 3^n\cdot M)+N|\\
&\geq |(M\cup 3^n\cdot M)|+|N|-1\geq |M|+|N|-1.
\end{align*}
If $|M|=1$ then, 
\begin{align*}
|S^2|&=|U^2|+|UV\cup VU|+|V^2|\\
&\geq 1+(1+|N|-1)+(4|N|-4)=5|N|-3\\
&\geq 4.5(1+|N|)-7=4.5|S|-7\geq \frac{7}{2}k-6.  
\end{align*}

\noindent Next assume that $|M|\geq 2.$\\
Now, consider the case where $\ell(M^{*})\geq 2|M^{*}|-2.$\\
 Suppose that, $|M|=2.$ Then $M=\{a_0<a_1\},$ which implies that $d(M)=a_1-a_0$ and $M^{*}=\{0,1\}.$ Thus $\ell(M^{*})=1$ and $1=\ell(M^{*})\geq 2|M^{*}|-2=2,$ which is a contradiction. Hence $|M|\geq 3.$ Therefore $k=|M|+|N|\geq 3+1=4.$\\
 Since $d(M^{*})=1$, using Theorem \ref{LSS}, we have
 $$|U^2|=|M+M|=|M^{*}+M^{*}|\geq 3|M^{*}|-3=3|M|-3.$$ 
 Hence
 \begin{align*}
 |S^2|&=|U^2|+|UV\cup VU|+|V^2|\\
 &\geq (3|M|-3)+|M|+|N|-1+(4|N|-4)\\
 &=4k+|N|-8\geq 4k-8\geq \frac{7}{2}k-6.
 \end{align*} 
 Finally consider the case where $\ell(M^{*})< 2|M^{*}|-2.$\\
By Theorem \ref{LSS}, we can assume that $h_{M^{*}}=\ell^{*}+1-|M^{*}|$, therefore
 $$|U^2|=|M+M|=|M^{*}+M^{*}|\geq 2|M^{*}|-1+h_{M^{*}}=|M^{*}|+\ell^{*}=|M|+\ell^{*}.$$
Next, we estimate the size of $M\cap 3^n\cdot M.$ Note that all the common elements of $3^n\cdot M$ and $M$ lie in the interval $[\min(M),\max(M)]$ of length $\ell$ and the set $3^n\cdot M$ is included in an arithmetic progression of difference $3^nd>3d$. Hence     
$$|M\cap (3^n\cdot M)|\leq \frac{\ell}{3d}+1=\frac{\ell^{*}}{3}+1.$$
Therefore 
$$|M\cup (3^n\cdot M)|=|M|+|3^n\cdot M|-|M\cap (3^n\cdot M)|\geq 2|M|-\frac{\ell^{*}}{3}-1.$$
Hence 
\begin{align*}
 |S^2|
 &=|U^2|+|UV\cup VU|+|V^2|\\
 &\geq |M+M|+(|M\cup (3^n\cdot M)|+|N|-1)+(|N+3^n\cdot N|)\\
 &\geq (|M|+\ell^{*})+(2|M|-\frac{\ell^{*}}{3}-1+|N|-1)+4|N|-4\\
 &=3|M|+5|N|-6+\frac{2\ell^{*}}{3}\\
&\geq 3|M|+5|N|-6+\frac{|M^{*}|-1}{3}=\frac{10}{3}|M|+5|N|-\frac{19}{3}\geq \frac{7}{2}k-6.
 \end{align*}

\noindent  {\bf{Case 2}}: Next suppose that $m\geq 1.$\\
 Here $|U^2|=|M+3^m\cdot M|\geq 4|M|-4$, $|V^2|=|N+3^n\cdot N|\geq 4|N|-4.$ Also $|UV|=|N+3^m\cdot M|\geq |M|+|N|-1.$ Hence $|S^2|=|U^2|+|UV|+|V^2|\geq 4|M|-4+|M|+|N|-1+4|N|-4=5k-9\geq \frac{7}{2}k-6.$ 
This completes the proof.
\end{proof}
\begin{lemma}\label{l2}
Let $S\subseteq BS(1,3)$ be a finite set of size $k.$ Suppose $S=S_0\cup S_1\cup\cdots\cup S_t$; $t\geq 2.$ Then the followings hold:\\
(a) If $k_0=|S_0|\geq 2$ and $k_i=|S_i|=1$, for $1\leq i\leq t,$ then $|S^2|\geq \frac{7}{2}k-6.$\\
(b) If $k_t=|S_t|\geq 2$ and $k_i=|S_i|=1$, for $1\leq i\leq t-1,$ then $|S^2|\geq \frac{7}{2}k-6.$ 
\end{lemma}
\begin{proof}
(a) Let
$A_0=\{y_1<y_2<\cdots<y_{k_0}\}\subset \mathbb{Z}$ be a finite set of $k_0$ integers. Define the set $S_0=a^{A_0}=\{a^{y_1},a^{y_2},\ldots,a^{y_{k_0}}\}$ 
with $k_0\geq 2,$ and let $S_i=\{b^{m_i}a^{x_i}\}$; $1\leq i\leq t.$ \\
Now observe that $k\geq 4$ and for every $1\leq i\leq t$, $m_i\geq 0$,\\
$S_0S_i=b^{m_0+m_i}\{a^{x_i+3^{m_i}y_1},\ldots,a^{x_i+3^{m_i}y_{k_0}}\}$, $|S_0S_i|=k_0$,\\
$S_iS_0=b^{m_i+m_0}\{a^{y_1+3^{m_0}x_i},\ldots,a^{y_{k_0}+3^{m_0}x_i}\}$, $|S_iS_0|=k_0.$\\
We claim that $|S_0S_i\cup S_iS_0|\geq k_0+1.$\\
 Firstly, we will show that $S_0S_i\neq S_iS_0.$ If $S_0S_i=S_iS_0$, then $\{x_i+3^{m_i}y_1<\cdots<x_i+3^{m_i}y_{k_0}\}=
\{y_1+3^{m_0}x_i<\cdots<y_{k_0}+3^{m_0}x_i\}$ and thus $(3^{m_0}-1)x_i=(3^{m_i}-1)y_1=\cdots=(3^{m_i}-1)y_{k_0}$ which contradicts  $\{y_0<\cdots<y_{k_0}\}$ as $m_i\geq 1$ and $k_0\geq 2.$ Therefore, $S_0S_i\neq S_iS_0.$ 
 Note that
$S_0S_i\cup S_iS_0\subseteq b^{m_0+m_i}a^{\mathbb{Z}}$, $S_iS_t\subseteq b^{m_i+m_t}a^{\mathbb{Z}}$ for every $0\leq i\leq t.$ Moreover $S_0S_0=b^{2m_0}a^{A_0+3^{m_0}\cdot A_0}$, $|S_0S_0|=|A_0+3^{m_0}\cdot A_0|\geq 2|A_0|-1=2k_0-1.$  Note that the sets
 $S_0S_0$, $S_0S_1\cup S_1S_0$, $\ldots$, $S_0S_t\cup S_tS_0$, $S_1S_t,\ldots,S_tS_t$ are disjoint and included in $S^2.$ Hence
\begin{align*}
 |S^2|
 &\geq (|S_0S_0|)+(|S_0S_1\cup S_1S_0|) \cdots (|S_0S_t\cup S_tS_0|)+(|S_1S_t|+\cdots+|S_tS_t|)\\
 &\geq (2k_0-1)+(k_0+1)+\cdots+(k_0+1)+(1+1+\cdots+1)\\
 &=(2k_0-1)+t(k_0+1)+t\\
 &=4k_0+(t-2)k_0+2t-1\\
 &\geq 4k_0+2(t-2)+2t-1=4k_0+4t-5\\
 &=4k-5\geq 3k-5\geq \frac{7}{2}|S|-4\geq \frac{7}{2}|S|-6.  
 \end{align*}
 Hence $|S^2|\geq \frac{7}{2}|S|-6.$\\
(b) Let $A_t=\{y_1<\cdots<y_{k_t}\}\subseteq\mathbb{Z}$ be a finite set of $k_t\geq 2$ integers and let $S_i=b^{m_i}a^{x_i}$ for every $0\leq i\leq t-1.$ \\ Observe that
$S_t=b^{m_t}a^{A_t}=\{b^{m_t}a^{y_1},\ldots,b^{m_t}a^{y_{k_t}}\}$.
Note that for every $1\leq i\leq t-1$, we have $S_tS_i=b^{m_t+m_i}\{a^{x_i+3^{m_i}y_1},\ldots,a^{x_i+3^{m_i}y_{k_t}}\}$ and
$S_iS_t=b^{m_i+m_t}\{a^{y_1+3^{m_t}x_i},\ldots,\\ a^{y_{k_t}+3^{m_t}x_i}\}$. Thus for every $0\leq i\leq t-1,$ 
\begin{align}\label{eq6}
|S_tS_i\cup S_iS_t|\geq k_t+1
\end{align}
 for every $1\leq i\leq t-1.$ Also $|S_tS_t|\geq 4k_t-4$ and $|S_0S_t|=k_t\geq 2.$ 
\begin{align*}
|S^2|&\geq (|S_0S_0|+\cdots+|S_0S_t|)+(|S_1S_t\cup S_tS_1|+\cdots+|S_{t-1}S_t\cup S_tS_{t-1}|+|S_tS_t|)\\
&\geq (1+1+\cdots+1+k_t)+(t-1)(k_t+1)+4k_t-4\\
&= (t-1)+k_t+(t-1)k_t+(t-1)+4k_t-4\\
&=5k_t+(t-1)k_t+2t-5\\
&\geq 5k_t+4t-7\geq 4k-7\\
&\geq \frac{7}{2}|S|-6.
\end{align*}
\end{proof}

\begin{lemma}\label{l3}
Let $S\subseteq BS^{+}(1,3)$ be a finite non-abelian set of size $k=|S|\geq 2.$ Suppose that
 $S=S_0\cup S_1\cup\ldots\cup S_t$, where $|S_i|=1$ for all $i$ and $S_i=\{s_i\}=\{b^{m_i}a^{x_i}\}$ for $1\leq i\leq t$, $m_0<m_1<\cdots<m_t$ and $s_0=a^{x_0}.$  If the subgroup $\langle T=S\setminus \{s_0\}\rangle$ is abelian, then $|S^2|\geq 4k-4.$
\end{lemma}
\begin{proof}
Observe that the sets $T^2$, $s_0 T\cup Ts_0$ and $\{s_0^2\}$ are disjoint, because\\
(i) $s_0\notin T^2 $, since $s_0^2=(b^{m_0}a^{x_0})^2=b^{2m_0}a^{x_0+3^{m_0}x_0}$ and $T^2\subseteq \{b^ma^x:m\geq 2m_1\}$\\
(ii) $s^2\notin (s_0T\cup Ts_0)$, because $s_0\notin T$\\
(iii) $s_0\notin \langle T\rangle$, because $\langle T\rangle$ is abelian and $\langle S\rangle$ is non-abelian. This implies that $(s_0T\cup Ts_0)\cap T^2=\phi.$\\
Since, $|T|=t$ and if $s_i,\,s_j\in T$, then $s_is_j=b^{m_i+m_j}a^{3^{m_j}x_i+x_j}.$ Thus $|T^2|\geq |M_s\setminus \{m_0\}+M_s\setminus \{m_0\}|\geq 2|M_s\setminus \{m_0\}|-1=2t-1.$ To complete the proof, we will show that $s_0T$ and $Ts_0$ are distinct. 
By contradiction suppose that $s_0T\cap Ts_0\neq \phi$, where\\
 $s_0T=\{s_0s_1,\ldots,s_0s_t\}=\{b^{m_0+m_1}a^{x_1+3^{m_1}x_0},\ldots,b^{m_0+m_t}a^{x_t+3^{m_t}x_0}\}$ and \\
  $Ts_0=\{s_1s_0,\ldots,s_ts_0\}=\{b^{m_0+m_1}a^{x_0+3^{m_0}x_1},\ldots,b^{m_0+m_t}a^{x_0+3^{m_0}x_t}\}$.\\
As $s_0T\cap Ts_0\neq \phi$, therefore there exists some $1\leq i\leq t$ such that \\
  $s_0s_i=b^{m_0+m_i}a^{x_i+3^{m_i}x_0}=s_is_0=b^{m_0+m_i}a^{x_0+3^{m_0}x_i}.$ Thus 
\begin{align}\label{eq1}
(3^{m_i}-1)x_0=(3^{m_0}-1)x_i.
\end{align}
 Choose an arbitrary $1\leq j\leq t.$ Since $T$ is abelian, it follows that 
 $$s_js_i=b^{m_j+m_i}a^{x_i+3^{m_i}x_j}=s_is_j=b^{m_j+m_i}a^{x_j+3^{m_j}x_i}=s_is_j.$$ 
 Thus $(3^{m_i}-1)x_j=(3^{m_j}-1)x_i$,  $x_i=\frac{3^{m_i}-1}{3^{m_j}-1}x_j.$\\
  From \eqref{eq1}, 
 $$(3^{m_i}-1)x_0=(3^{m_0}-1)\frac{3^{m_i}-1}{3^{m_j}-1}x_j.$$ This implies $(3^{m_j}-1)x_0=(3^{m_0}-1)x_j$ and thus 
 $$s_0s_j=b^{m_0+m_j}a^{{x_j}+3^{m_j}x_0}=b^{m_0+m_j}a^{{x_0}+3^{m_j}x_j}=s_js_0.$$ Therefore, $s_0$ commutes with every element of $T$, which is a contradiction to our assumption that $\langle T\rangle$ is abelian and $\langle S\rangle$ is non-abelian.
Consequently,
\begin{align*}
|S^2|&\geq |T^2|+|s_0T\cup Ts_0|+|\{{s_0}^2\}|\\
&=|T^2|+|s_0T|+ |Ts_0|+|\{{s_0}^2\}|\\
&\geq (2t-1)+t+t+1=4t=4|S|-4.
\end{align*} 
\end{proof}
\begin{lemma}\label{l4}
Let $S\subseteq BS(1,3)$ be a finite set of cardinality $k=|S|\geq 2$. Suppose that $S$ is a disjoint union, $S=V_1\cup V_2\cup\cdots\cup V_t$ of $t$ subsets, where $V_i=\{s_i\}=\{b^{m_i}a^{x_i}\}$ and  $|V_i|=1.$ If $S$ is a non-abelian set and $1\leq m_1<m_2<\cdots<m_t$, then $|S^2|\geq 4k-4.$
\end{lemma}
\begin{proof}
 We proceed by induction on $t$ and clearly $t=k\geq 2.$  If $t=2,$ then $S=\{s_1,s_2\}$
 and since ${s_1}^2\neq {s_2}^2$ and $s_1s_2\neq s_2s_1$, it follows that $|S^2|=4=4|S|-4,$ which is required.\\
 For the inductive step, let $t\geq 3$ be an integer, and assume that Lemma \ref{l4} holds for each set $T\subseteq BS(1,3)$ of size $2\leq |T|\leq t-1.$ Denote $S'=S\setminus \{s_1\}$. By above part of this Lemma, we can suppose $\langle S'\rangle$ is non-abelian. We will continue the proof by two cases.\\
 {\bf{Case 1:}} Let us assume that $s_1s_2=s_2s_1.$ \\
 Choose $n\geq 2$ maximal such that the set $S^{^{*}}=\{s_1,s_2,\ldots,s_n\}$ is abelian. Note that $n<t$,  because $S$ is a non-abelian set and $s_{n+1}\notin \langle S^{^{*}}\rangle.$ Moreover, $s_1s_{n+1}\notin {S'}^2$, since otherwise $s_1s_{n+1}=s_us_v$ for some $1\leq u,v\leq t$ and hence 
 $b^{m_1+m_{n+1}}=b^{m_u+m_v}$ implies $m_1<m_u$ and $m_v<m_{n+1}$, $1<u,v<n+1$ and $s_{n+1}\in \langle S^{^{*}}\rangle$, a contradiction. Similarly, $s_{n+1}s_1\notin {S'}^2$. \\
If $s_1s_{n+1}\neq s_{n+1}s_1$, then $\{s_1^2,s_1s_2,s_1s_{n+1},s_{n+1}s_1\}\subseteq S^2\setminus S'^2$ and from induction hypothesis for $S'$:\\
  $|S^2|\geq |{S'}^2|+|\{s_1^2,s_1s_2,s_1s_{n+1},s_{n+1}s_1\}|\geq (4|S'|-4)+4=4|S|-4.$ Next we will complete the proof of Lemma by showing that if
 \begin{align} \label{eq2}
 s_1s_{n+1}=s_{n+1}s_1
 \end{align} 
 then
  $s_js_{n+1}=s_{n+1}s_j$ for every $1\leq j\leq n,$ which contradicts the maximality of $n.$\\
 Let us denote $m=m_{n+1}$, $x=x_{n+1}$ and $s_{n+1}=b^{m}a^{x}.$ By equation \eqref{eq2}
 \begin{align*}
 s_1s_{n+1}&=(b^{m_1}a^{x_1})(b^ma^{x})=b^{m_1+m}a^{x+3^{m}x_1}\\
 &=s_{n+1}s_1=(b^ma^{x})b^{m_1}a^{x_1}=b^{m+m_1}a^{x_1+3^{m_1}x}
 \end{align*}
 and thus 
 \begin{align}\label{eq3}
 (3^{m_1}-1)x=(3^{m}-1)x_1.
 \end{align} 
 Choose an arbitrary $1\leq j\leq n.$ Using $s_1s_j=s_js_1,$ we get $x_1=\frac{3^{m_1}-1}{3^{m_j}-1}x_j$. By equation \eqref{eq3}, $(3^{m_1}-1)x=(3^m-1)\frac{3^{m_1}-1}{3^{m_j}-1}x_j.$ Since $m_1\geq 1$, $(3^{m_j}-1)x=(3^m-1)x_j.$  Hence $s_js_{n+1}=s_{n+1}s_j,$ which is a contradiction.\\
 {\bf{Case 2:}} $s_1s_2\neq s_2s_1.$ We claim that either $s_1s_3\neq {s_2}^2$ or $s_3s_1\neq {s_2}^2.$ Indeed if $s_1s_3={s_2}^2$ and $s_3s_1={s_2}^2$ then  $s_1s_3=s_3s_1$ and $s_1{s_2}^2={s_2}^2s_1.$ As $BS(1,3)$ is orderable, then $s_1s_2=s_2s_1$ which is a contradiction. Hence our claim is proved. Thus $|\{s_1s_3,s_3s_1\}\setminus \{{s_2}^2\}|\geq 1.$\\
 Since, $\{{s_1}^2,s_1s_2,s_2s_1\}\subseteq S^2\setminus {S'}^2$ and $\{s_1s_3,s_3s_1\}\setminus \{{s_2}^2\}\subseteq S^2\setminus {S'}^2.$ Then by induction hypothesis for $S'$, that $|S^2|\geq |{S'}^2|+|\{{s_1}^2,s_1s_2,s_2s_1,s_1s_3,s_3s_1\}\setminus {S'}^2|\geq (4|S'|-4)+4=4|S|-4.$ This completes the proof of Lemma.  
  \end{proof}
  \section{Proof of Main Theorem}
\begin{theorem}\label{thm3}
Let $S$ be a finite non-abelian subset of $BS^{+}(1,3)$ of size $k=|S|\geq 3.$ \\
(I) If $S=S_0\cup S_1\cup\cdots\cup S_t,$ where $t\geq 1,$  then $|S^2|\geq \frac{7}{2}|S|-6.$\\
(II) If $S=b^ra^A$, where $A$ $(|A|=k)$ is a finite set of integers containing $0$ and $|S^2|=(t+2)k-2t+h$, where $c_3(A)=t.$  Then $A$ is contained in an arithmetic progression of size $k+h<\frac{5}{2}k-t(k-2)-6.$
 \end{theorem}
\begin{proof} 
 (I) 
  If $t=1$, then result holds from Lemma \ref{l1}. For the induction step, let $t\geq 2$ be an integer. Assume that the result holds for any non-abelian finite set $T\subseteq BS(1,3),$ which lies in $u$ distinct cosets of $\langle a\rangle =a^{\mathbb{Z}}$, where $2\leq u<t$.\\
  Denote $S^{*}=S\setminus S_t$, $k^{*}=|S^{*}|=k-k_t.$ As $S^{*}$ generates a non-abelian group, then our induction hypothesis implies that\\
 $|(S^{*})^2|\geq \frac{7}{2}k^{*}-6.$ It is sufficient to show that \\
  $|{S_t}^2\cup S_tS_{t-1}\cup S_{t-1}S_t|\geq \frac{7}{2}k_t.$
Now we examine four cases:\\
{\bf{Case 1:}} Assume that either $k_t\geq 2$ or $k_{t-1}\geq 2.$\\
Recall that $0\leq m_0<m_1<\cdots<m_t.$ As $t\geq 2$ implies that $m_t\geq 2.$ Thus by Corollary \ref{cor1}, ${S_t}^2\geq \max\{4k_t-4,1\}\geq 3k_t-2.$\\  
We will examine two subcases:\\
{\bf{Subcase 1:}} Let us assume that $k_t+2k_{t-1}\geq 6$, then 
\begin{align*}
|{S_t}^2\cup S_tS_{t-1}\cup S_{t-1}S_t| &\geq |{S_t}^2|+|S_tS_{t-1}\cup S_{t-1}S_t|\geq (3k_t-2)+(k_t+k_{t-1}-1)\\
& \geq 4k_t+k_{t-1}-3\geq \frac{7}{2}k_t.
\end{align*}
If $k_j=|S_j|\geq 2$ for some $j$, $0\leq j\leq t-1$, then by Lemma \ref{l1}, $S^{*}$ generates a non-abelian group. By induction hypothesis, 
\begin{align}
|S^2| &\geq |(S^{*})^2|+|{S_t}^2\cup S_tS_{t-1}\cup S_{t-1}S_t|\nonumber\\
 &\geq \Big(\frac{7}{2}k^{*}-6\Big)+\frac{7}{2}k_t=\frac{7}{2}k-6.
 \end{align} 
If $k_j=1$ for all $0\leq j\leq t-1$, then $k_t\geq 6-2k_{t-1}=4.$ In this case, result follows from Lemma \ref{l2}. \\
{\bf{Subcase 2:}} Next we assume that $k_t+2k_{t-1}< 6.$ We have following possibilities for the values of $k_t$ and $k_{t-1}$: 
$$k_t=3, \,k_{t-1}=1$$ 
 $$k_t=2,\, k_{t-1}=1$$ 
 $$k_t=1,\,k_{t-1}=2.$$\\
 
\noindent  For $k_t=3$ and $k_{t-1}=1$, then by Corollary \ref{cor1}, ${S_t}^2\geq 4k_t-4$ and therefore
\begin{align*}
 |{S_t}^2\cup S_tS_{t-1}\cup S_{t-1}S_t| &\geq |{S_t}^2|+|S_tS_{t-1}\cup  S_{t-1}S_t |\geq (4k_t-4)+(k_t+k_{t-1}-1)\\
& \geq 5k_t-5\geq \frac{7}{2}k_t.
\end{align*}
 
\noindent If $k_j = |S_j| \geq 2$ for a certain index $j$, $0 \leq j \leq t-1$, then, according to Lemma \ref{l1}, the set $S^{*}$ produces a non-abelian group. This conclusion is supported by the induction hypothesis.

 \noindent If $k_j=1$ for all $0\leq j\leq t-1,$ then result follows from Lemma \ref{l4}, in view of $k_t=3.$\\

\noindent For $k_t=2$ and $k_{t-1}=1$, then write $S_{t-1}=\{b^ua^x\}$, $S_t=\{b^va^y,b^va^z\}$, where $y<z$ are integers. $S_{t-1}S_t=b^{u+v}\{a^{3^vx+y},{a^{3^vx+z}}\}$ and $S_{t}S_{t-1}=b^{u+v}\{a^{3^uy+x},a^{3^uz+x}\}.$ Observe that $S_{t-1}S_t\neq S_tS_{t-1}.$ Indeed, if $S_{t-1}S_t=S_tS_{t-1}$, then $3^vx+y=3^uy+x$ and $3^vx+z=3^uz+x$ which implies $(3^u-1)y=(3^v-1)y=(3^u-1)z$, contradicts as $y<z.$ Hence 
\begin{align*}
|{S_t}^2\cup S_tS_{t-1}\cup S_{t-1}S_t| &\geq |{S_t}^2|+|S_tS_{t-1}\cup  S_{t-1}S_t|\\
 &\geq (3k_t-2)+3\\
& =4+3 =\frac{7}{2}k_t.
\end{align*}

\noindent If there is $0\leq j\leq t-1$ such that $k_j=|S_j|\geq 2$, then by Lemma \ref{l1}, $S^{*}$ generates a non-abelian group. And by induction hypothesis, 
\begin{align}
|S^2| &\geq |(S^{*})^2|+|{S_t}^2\cup S_tS_{t-1}\cup S_{t-1}S_t|\nonumber\\
 &\geq \Big(\frac{7}{2}k^{*}-6\Big)+\frac{7}{2}k_t=\frac{7}{2}k-6.
 \end{align} 

\noindent If $k_j=1$ for all $0\leq j\leq t-1$, then result follows from Lemma \ref{l4}, in view of $k_t=2.$\\ 

\noindent For  $k_t=1$ and $k_{t-1}=2$, then we can write $S_{t-1}=\{b^ua^y,b^ua^z\}$, $S_{t}=\{b^va^x\}$ where $1\leq u=m_{t-1}<v=m_t$ and $x,y,z$ are integers and $y<z$. $S_{t-1}S_t=b^{u+v}\{a^{3^vy+x},a^{3^vz+x}\}$ and $S_{t}S_{t-1}=b^{u+v}\{a^{3^ux+y},a^{3^ux+z}\}.$ Note that $S_{t-1}S_t\neq S_tS_{t-1}.$ Indeed, if $S_{t-1}S_t=S_tS_{t-1}$, then  $(3^u-1)y=(3^v-1)x=(3^u-1)z$, which contradicts as $y<z$ and $v\geq 1.$ Hence 
\begin{align*}
|{S_t}^2\cup S_tS_{t-1}\cup S_{t-1}S_t| &\geq |{S_t}^2|+|S_tS_{t-1}\cup  S_{t-1}S_t|\\
& \geq 1+3> \frac{7}{2}k_t.
\end{align*}
By Lemma \ref{l1}, $S^{*}$ generates a non-abelian group. And by induction hypothesis, 
\begin{align}
|S^2| &\geq |(S^{*})^2|+|{S_t}^2\cup S_tS_{t-1}\cup S_{t-1}S_t|\nonumber\\
 &\geq \Big(\frac{7}{2}k^{*}-6\Big)+\frac{7}{2}k_t=\frac{7}{2}k-6.
 \end{align} 
{\bf{Case 2:}} Assume that $k_t=k_{t-1}=\cdots=k_1=1$ and $k_0\geq 2.$ In this case, the result follows from Lemma \ref{l3}.\\
{\bf{Case 3:}} Assume that $k_t=k_{t-1}=1$ and there is $1\leq j\leq t-2$ such that $k_j\geq 2$ and $k_i=1$ for every $i\in \{j+1,\ldots,t\}.$\\
Let $S_j=b^{m_j}a^{A_j}=\{b^{m_j}a^{y_1},b^{m_j}a^{y_2},\ldots,b^{m_j}a^{y_{k_j}}\}$ and $S_i=\{b^{m_i}a^{x_i}\}$ for every $i\in \{j+1,\ldots,t\}.$ Clearly $|S_jS_i|=|S_iS_j|=k_j$ and using the proof of Lemma \ref{l3}, we get
\begin{align}\label{eq5}
 |S_jS_i\cup S_iS_j|\geq k_{j}+1
 \end{align}
  for every $i\in \{j+1,\ldots,t\}.$ Note that $k_j\geq 2,$ so by Lemma \ref{l1}, the set ${S_j}^{*}=S_0\cup S_1\cup\cdots\cup S_j$ is non abelian. By applying induction hypothesis to ${S_j}^{*}$ and by \eqref{eq5}, we obtain
  \begin{align*}
  |S^2|&\geq |{s_j}^{*}{s_j}^{*}|+\sum_{u=j+1}^t|S_jS_u\cup S_uS_j|+\sum_{u=j+1}^t|S_uS_t|\\
  & \geq \Big(\frac{7}{2}|{s_j}^{*}|-6\Big)+(k_j+1)(t-j)+(t-j)\\
  &= \Big(\frac{7}{2}|{s_j}^{*}|-6\Big)+(k_j+2)(t-j)\geq\Big(\frac{7}{2}{s_j}^{*}-6\Big)+4(t-j)\\
  &=\Big(\frac{7}{2}|{s_j}^{*}|-6\Big)+4(k-|{s_j}^{*}|)\\
  &=4k-6-\frac{1}{2}|{s_j}^{*}|\geq \frac{7}{2}k-6,
  \end{align*}
  which is required.\\
  {\bf{Case 4:}} Assume that $k_i=1$ for every $0\leq i\leq t.$ If the set $S'=S\setminus\{s_0\}$ is abelian, then Lemma \ref{l3} implies that $|S^2|\geq 4k-4\geq \frac{7}{2}k-6,$ which is required. Next, we  assume that $S'=S_1\cup S_2\cup \cdots \cup S_t$ is non-abelian. Since $t\geq 2$, it follows that $k=t+1\geq 3$ and Lemma \ref{l4} implies that $|{S'}^2|\geq 4|S'|-4.$ Moreover, $\{{s_0}^2,s_0s_1\}\subseteq S^2\setminus {{S'}^2}.$ We consider two cases:\\
(a) If $k=|S|\geq 4,$ then $|S^2|\geq |{S'}^2|+2\geq 4(k-1)-4+2=4k-6\geq \frac{7}{2}k-6,$ as required.\\
(b) If $k=3$, then $S=\{s_0,s_1,s_2\}$, $S'=\{s_1,s_2\}$, $s_1s_2\neq s_2s_1$, $|{S'}^2|=4$ and $S^2=\{{s_0}^2,s_0s_1,s_1s_0,s_0s_2,s_2s_0\}\cup {S'}^2.$ We claim that either $s_0s_2\neq {s_1}^2$ or $s_2s_0\neq {s_1}^2$. Indeed, if $s_0s_2={s_1}^2$ and $s_2s_0={s_1}^2$ then $s_0s_2=s_2s_0$ and hence ${s_1}^2s_2=s_2{s_1}^2.$ As $BS(1,3)$ is orderable so this equality implies that $s_1s_2=s_2s_1$, which is contradiction. We conclude that $|S^2|\geq |\{{s_0}^2,s_0s_1,s_0s_2,s_2s_0\}\setminus {S'}^2|+|{S'}^2|\geq 3+4=7>\frac{7}{2}|S|-6.$ This completes the proof of part (I).\\
(II) Let $S$ be a finite set satisfying $|S^2|=(t+2)k-2t+h<\frac{7}{2}k-6$, and by above $S=S_0=b^{m_0}a^{A_0}$: $m_0\geq 0.$ As the set $S$ is non abelian so $m_0\geq 1.$ But $m_0=1$ because if $m_0\geq 2$, then Corollary \ref{cor1} implies that $|S^2|\geq 4k-4>\frac{7}{2}k-6$ which contradicts our hypothesis. Hence $S=S_0=ba^A$ and Theorem \ref{thm2} is applicable. By Theorem \ref{thm1} if $c_3(A)=t$ and $|S^2|=(t+2)k-2t+h<\frac{7}{2}k-6$ then $h<\frac{3}{2}k-t(k-2)-6$. And $A$ is contained in arithmetic progression of size $k+h<k+\frac{3}{2}k-t(k-2)-6=\frac{5}{2}k-t(k-2)-6.$ Hence the proof of Theorem \ref{thm3} is complete.
\end{proof}


\begin{thebibliography}{88}


	



\bibitem{C} A. L. Cauchy, Recherches sur les nombres, J. \'{E}cole Polytechnique 9 (1813), 99--123.

\bibitem{D} H. Davenport, On the addition of residue classes, J. London Math. Soc. 10 (1935) 30--32.

\bibitem{EK98} S. Eliahou and and M. Kervaire, \emph{Sumsets in vector spaces over finite fields}, J. Number Theory 71 (1998), 12--39.

\bibitem{EK12} S. Eliahou and M. Kervaire, \emph{ Some Results on minimal sumset sizes in finite non-abelian groups}, Journal of  Number Theory 124 (2007), 234--247.

\bibitem{EK13} S. Eliahou and M. Kervaire, \emph{The small sumset property for solvable finite groups}, European Journal of Combinatorics 27 (2006), 1102--1110.

\bibitem{EK14} S. Eliahou and M. Kervaire, \emph{Minimal sumsets in infinite abelian groups}, J. Algebra 287 (2005), 449--457.
 
 \bibitem{EK15} S. Eliahou and M. Kervaire, \emph{Sumsets in dihedral groups}, European Journal of Combinatorics 27 (2006), 617--628.

 \bibitem{EKP} S. Eliahou, M. Kervaire, A. Plagne \emph{Optimally small sumsets in finite abelian groups}, J. Number Theory 101 (2003), 338--348.
 
\bibitem{freiman1} G. A. Freiman, On the addition of finite sets, \emph{I. Izv. Vyss. Ucebn. Zaved. Matematika} 13 (1959), 202--213.

\bibitem{freiman} G. A. Freiman, M. Herzog, P. Longobardi, M. Maj, and Y. V. Stanchescu, \emph{Direct and inverse problems in additive number theory and in non-abelian group theory}, European J. Combin. 40 (2014), 42--54. 

\bibitem{freiman15} G. A. Freiman, M. Herzog, P. Longobardi, M. Maj, and Y. V. Stanchescu, \emph{A small doubling structure theorem in a Baumslag-Solitar group}, European J. Combin. 44 (2015), 106--124.



\bibitem{LS} V.F. Lev, P.Y. Smeliansky, \emph{On addition of two distinct sets of integers}, Acta Arith. 70 (1) (1995) 85–-91.

\bibitem{P} A. Plagne, \emph{Additive number theory sheds extra light on the Hopf-Stiefel $\circ$ function}, L'Enseignement Mathematique 49 (2003), 109--116.

\bibitem{RS} R. Kaur, S. Singh, \emph{On direct and inverse problems related to some dilated sumsets},  Comptes Rendus Mathématique(accepted).




 
 
 \bibitem{nathanson96} M. B. Nathanson, \emph{Additive Number Theory: Inverse Problems and the Geometry of Sumsets}, Springer, 1996.

\bibitem{S} Y.V. Stanchescu, \emph{On addition of two distinct sets of integers}, Acta Arith. 75 (2) (1996) 191–-194. 
\end{thebibliography}
\end{document}